\documentclass[a4paper,11pt]{amsart}
\usepackage[pdftex]{color,graphicx}
\usepackage{amssymb}
\usepackage{amsfonts}
\usepackage{esint}
\usepackage{amsmath}
\usepackage{amsrefs}
\usepackage{epsfig}
\usepackage{amsthm}
\usepackage{color}
\usepackage[]{epsfig}
\usepackage[]{pstricks}
\usepackage[utf8]{inputenc}
\newtheorem{theorem}{Theorem}[section]
\newtheorem{proposition}{Proposition}[section]
\newtheorem{lemma}{Lemma}[section]
\newtheorem{corollary}{Corollary}[section]

\theoremstyle{definition}
\newtheorem{definition}{Definition}[section]
\newtheorem{example}{Example}[section]

\theoremstyle{remark}
\newtheorem{remark}{Remark}[section]

\newcommand{\R}{\mathbb{R}}

\newcommand{\s}{\mathbb{S}}

\newcommand{\LL}{\mathcal{L}}
\newcommand{\CC}{\mathcal{C}}
\newcommand{\n}{\nabla}
\newcommand{\ran}{\rangle}
\newcommand{\lan}{\langle}
\newcommand{\ve}{\varepsilon}

\newcommand{\sss}{\sigma}
\newcommand\raisepunct[1]{\,\mathpunct{\raisebox{0.5ex}{#1}}}
\DeclareMathOperator{\hess}{hess}
\DeclareMathOperator{\hesss}{Hess}

\DeclareMathOperator{\di}{div}
\DeclareMathOperator{\tr}{trace}

\DeclareMathOperator{\dist}{dist}

\numberwithin{equation}{section}
\title[Gap theorems for self-shrinkers of $r$-mean curvature flows]{Gap theorems for complete self-shrinkers\\ of $r$-mean curvature flows}
\author{Hil\'ario Alencar, G. Pacelli Bessa \and Greg\'orio Silva Neto}
\date{\today}
\dedicatory{Dedicated to Paolo Piccione by the occasion of his 60th birthday}

\address{Instituto de Matem\'atica,
Universidade Federal de Alagoas,
Macei\'o, 57072-900, Brazil}
\email{hilario@mat.ufal.br}

\address{Departamento de Matem\'atica,
Universidade Federal do Ceará,
Fortaleza, 60455-760, Brazil}
\email{bessa@mat.ufc.br}

\address{Instituto de Matem\'atica,
Universidade Federal de Alagoas,
Macei\'o, 57072-900, Brazil}
\email{gregorio@im.ufal.br}

\begin{document}
\footnotetext{The authors were partially supported by the National Council for Scientific and Technological Development - CNPq of Brazil [Grant: 303118/2022-9 to H. Alencar,
303057/2018-1, 402563/2023-9 to G. P. Bessa and 304032/2020-4 to G. Silva Neto]. G. Silva Neto was also partially suported by the Alagoas Research Foundation [Grant: E:60030.0000000161/2022].}

\begin{abstract}
 In this paper, we prove gap results for complete self-shrinkers of the $r$-mean curvature flow involving a modified second fundamental form. These results extend previous results for self-shrinkers of the mean curvature flow due to Cao-Li and Cheng-Peng. To prove our results we show that, under suitable curvature bounds, proper self-shrinkers are parabolic for a certain second-order differential operator which generalizes the drifted Laplacian and, even if is not proper, this differential operator satisfies an Omori-Yau type maximum principle.
\end{abstract}

\maketitle

\section{Introdution and main results}
\noindent Let $X\colon\Sigma^n \rightarrow \mathbb{R}^{n+1}$ be an isometric immersion of a $n$-dimensional Riemannian manifold into the Euclidean space  $\R^{n+1}$. Let $A\colon T\Sigma^n \to T\Sigma^n$ be its shape operator given by $A(Y)=-\overline{\nabla}_{Y}N,$ $Y\in T\Sigma^n,$ where $N$ is a locally defined normal vector field on $\Sigma^n$ and $\overline{\nabla}$ is the Levi-Civita connection of $\mathbb{R}^{n+1}.$ The shape operator $A$ is symmetric and its eigenvalues $k_1, \ldots, k_{n}$ are the principal curvatures of the hypersurface $\Sigma^n$. The elementary symmetric functions of the principal curvatures, called the $r$-mean curvatures of $\Sigma$, are defined by
\begin{equation}\label{eq1.1}
 \left\{\begin{aligned}
 \sss_0&=1,\\ 
 \sss_r&=\displaystyle{\sum_{i_1<\cdots<i_r}k_{i_1}\cdots k_{i_r}},\quad \mbox{for}\quad 1\leq r\leq n,\\
 \sss_r&=0,\,\,\mbox{for}\,\, r>n.\\
\end{aligned}\right.
\end{equation}
These functions appear naturally in the characteristic polynomial of $A,$ since
\begin{equation*}\label{char}
\det(A-tI)= \sss_n-\sss_{n-1}t+\sss_{n-2}t^2 -\cdots +(-1)^nt^n = \sum_{j=0}^n(-1)^j\sss_{n-j}t^j. 
\end{equation*}
Observe that
\[\sss_1=k_1+k_2+\cdots+k_n, \ \sss_2=\sum_{i<j}k_ik_j, \ \mbox{and}\ \sss_n=k_1k_2\cdots k_n
\]
are respectively  the mean curvature $H$, the scalar curvature, and the Gauss-Kronecker curvature $K$.
In this article, we will assume that $\Sigma^n$ has a continuous globally defined unit normal vector field $N$.

A family of immersions $X\colon\Sigma^n\times[0,T)\rightarrow\R^n$ is said to be a solution of the $r$-mean curvature flow if satisfies the initial value problem
\begin{equation}\label{flow}
\left\{\begin{aligned}
\dfrac{\partial X}{\partial t}(x,t)&=\sss_r(k_1(x,t),\ldots,k_n(x,t))N(x,t),\\
X(\cdot\,,0)&=\,X_0.
\end{aligned}\right.
\end{equation} 
Here, $k_1(x,t),\ldots,k_n(x,t)$ are the principal curvatures of the immersions $X_t\colon=X(\cdot,t),$ $N(\cdot,t)$ are their normal vector fields. We are adopting the convention on the normal $N$ such that in the spheres and in every closed hypersurface the normal points inward (i.e., in the direction of the region bounded by the hypersurface). With this convention, in the $n$-dimensional round sphere $\s^n(R)$ of radius $R,$ $X=-RN,$ the principal curvatures are positive and, for spheres and cylinders $\s^m(R)\times\R^{n-m}$, $1\leq m\leq n-1,$ the flow contracts.

The $r$-mean curvature flow is a natural generalization of the well-known mean curvature flow ($r=1$) and the Gaussian curvature flow ($r=n$) that has been widely investigated in the last four decades. Beside these cases,  the $r$-mean curvature flow can be found in the works of several authors, as \cite{AS2010}, \cite{AMC2012},  \cite{AW2021}, \cite{BS2018}, \cite{CRS2010}, \cite{Chow1987}, \cite{GLM2018},  \cite{GLW2017},   \cite{LWW2021},  \cite{LSW2020}, \cite{Urbas1990}, \cite{Urbas1999} and  \cite{Z2013}.

A solution $X(\cdot, t)$ of \eqref{flow} is said homothetic if there exists a  positive $\mathcal{C}^1$-function $\phi\colon [0,T)\to\R$ such that $\phi(0)=1$ and
\begin{equation}\label{pre-SS}
X (x,t)=\phi(t)X_0(x),\,\, \forall x \in \Sigma^n.
\end{equation}
If $\phi$ is a decreasing function, $\Sigma^n$ shrinks homothetically under the action of the flow, then $\Sigma^n$ is called a self-shrinker. It can be easily proven, after rescaling, that if $\Sigma^n$ is a self-shrinker of the $r$-mean curvature flow, then the $r^{\rm th}$-symmetric function $\sigma_r$  of  $\Sigma^n$ satisfies the equation
\begin{equation}\label{self-S}
\sss_r=-\lan X,N\ran, \quad 1\leq r \leq n,
\end{equation}
where $X$ is the position vector of $\Sigma^n$ in $\R^{n+1}.$

To state the results of this paper, we recall the definition of the Newton transformations, which can be understood as a natural generalization of the second fundamental form related to the symmetric functions $\sss_r$. Inspired by the characteristic polynomial of $A$ we define the $r$-th Newton transformation $P_r\colon T\Sigma^n\to T\Sigma^n$, $0\leq r\leq n,$ as the polynomial
\begin{equation}\label{Pr-pol}
    P_r = \sss_rI-\sss_{r-1}A+\sss_{r-2}A^2 -\cdots +(-1)^rA^r = \sum_{j=0}^r(-1)^j\sss_{r-j}A^j,
\end{equation}
where we are setting $P_0=I.$ It can be  seen that the Newton transformations satisfy the recurrence relation 
\begin{equation}\label{Pr-rec}
P_r=\sss_r I - P_{r-1}A,\quad 1\leq r\leq n,
\end{equation}
and, by the Cayley-Hamilton theorem, we have that $P_n=0.$


In the context of Differential Geometry, the Newton transformations $P_r$ first appeared in the work of Reilly \cite{Reilly}, in the expressions of the variational integral formulas for functions  $f(\sss_0,\ldots, \sss_n)$ of the elementary symmetric functions $\sss_i$'s.
Since we are assuming that $\Sigma^n$ has a global choice of $N$ we have that $P_r$ globally defined. 

In the following, we present some basic examples of self-shrinkers of the $r$-mean curvature flow.

\begin{example}\label{example}
Hyperplanes passing through the origin, the round sphere $\s^n\left(\delta_n(r)\right)$ of radius $\delta_n(r)=\binom{n}{r}^{\frac{1}{r+1}},$ and the cylinders $\s^{m}\left(\delta_m(r)\right)\times\R^{n-m}$ in $\R^{n+1},$ $r\leq m\leq n-1,$ are self-shrinkers of the $r$-mean curvature flow. For hyperplanes passing through the origin, we have the $r^{\ rm th}$-symmetric function $\sss_r=0=-\lan X,N\ran.$ On the other hand, in $\s^{m}\left(\delta_m(r)\right)\times\R^{n-m},$  $m\in\{0,1,\ldots, n\}$,  the principal curvatures are $k_1=(1/\delta_{m}(r))=\binom{m}{r}^{-\frac{1}{r+1}}$ with multiplicity $m$ and $k_2=0$ with multiplicity $n-m.$ This gives that
\begin{equation}\label{sigma-p}
\sss_p=\binom{m}{p}\binom{m}{r}^{-\frac{p}{r+1}}, \quad 0\leq p \leq n,
\end{equation}
where we are using the convention that $\binom{m}{k}=0$ if $k>m.$ Therefore,
\begin{equation}\label{SS-cyl}
\sss_r=\binom{m}{r}\binom{m}{r}^{-\frac{r}{r+1}}=\binom{m}{r}^{\frac{1}{r+1}}=-\lan X,N\ran,
\end{equation}
since $\lan X, N\ran$ equals the negative radius in spheres and cylinders. Notice that \eqref{SS-cyl} holds only for $r\leq m\leq n.$ Indeed, if $m<r,$ then $\sigma_r=0$, and thus, the respective cylinder does not satisfy the self-shrinker equation.
\end{example}


 Observe that for each $x\in \Sigma^n$ the linear operator $P_{r-1}(x)\colon T_x\Sigma^n \to T_x\Sigma^n$ is symmetric hence  $T_x\Sigma^n$ has a basis formed with eigenvectors of $P_{r-1}(x)$ associated to eigenvalues $\lambda_1(x)\leq  \lambda_2(x)\leq  \ldots \leq  \lambda_n(x)$.  Moreover, since  $P_{r-1}$  is a polynomial in $A$ we have that $AP_{r-1}=P_{r-1}A$ and $A$ and $P_r-1$ are simultaneously diagonalizable. The operator $P_{r-1}$ is positive semidefinite if   $\lambda_i(x)\geq 0$, $\forall\, x\in \Sigma^n$. The square root of $P_{r-1},$ as the only linear operator $\sqrt{P_{r-1}}\colon T\Sigma^n\to T\Sigma^n$ such that $(\sqrt{P_{r-1}})^2=P_{r-1}.$  Let $\{e_1,\ldots,e_n\}\subset T\Sigma^n$ be an orthonormal frame of eigenvectors of $A$ corresponding to the eigenvalues $\{k_1, k_2,\ldots, k_n\}$. Letting $A_i\colon e_i^{\perp} \to e_i^{\perp}$ to be the restriction of $A$ to $e_i^{\perp}$,  $i=1,\ldots,n$,  we have that  the eigenvalues $\lambda_i$ of $P_{r-1}$ are  the symmetric functions $\lambda_i=\sss_{r-1}(A_i)=\sss_{r-1}(k_1,k_2,\ldots, k_{i-1}, k_{i+1}, \ldots, k_n)$ associated to $A_i,$ see \cite{BC}, p.279. This gives that $\sqrt{\sss_{r-1}(A_i)},$ $i=1,\ldots,n,$ are the eigenvalues of $\sqrt{P_{r-1}}$. 
 
 In our main results, we will consider gap theorems involving the trace norm of the modified second fundamental form $\sqrt{P_{r-1}}A.$ 
\begin{eqnarray}
\Vert\sqrt{P_{r-1}}A\Vert^2&=& \tr\left((\sqrt{P_{r-1}}A)^{t}\cdot\sqrt{P_{r-1}}A)\right)= \tr\left(P_{r-1}A^2\right) \nonumber \\
&=& \sum_{j=1}^{n}\langle P_{r-1}A^2(e_j), e_j\rangle=\sum_{j=1}^n\sss_{r-1}(A_j)k_j^2.\nonumber 
\end{eqnarray}
Here $(\sqrt{P_{r-1}}A)^{t}=(\sqrt{P_{r-1}}A)$ since the operator $\sqrt{P_{r-1}}A$ is symmetric.
The quantity $\|\sqrt{P_{r-1}}A\|^2$ is quite natural in Differential Geometry in the context of $\sss_r.$ It appears in the formula of the second variation of $\int_\Sigma \sss_r d\Sigma,$ see \cite{AdCE}, p.207, Proposition 4.4, p.284 of \cite{BC}, and Theorem B, p.407 of \cite{Reilly}. It also appears in the definition of $r$-special hypersurface in \cite{AdCE}, p.203-204, as well as in the gap theorems of Alencar, do Carmo and Santos, see \cite{AdCS}, and Alias, Brasil and Sousa, see \cite{ABS}.

In the next, we calculate $\|\sqrt{P_{r-1}}A\|^2$ for the basic examples. Clearly, hyperplanes satisfy $\|\sqrt{P_{r-1}}A\|^2=0.$ In $\s^{m}\left(\delta_{m}(r)\right)\times\R^{n-m},$ for $r\leq m\leq n,$ we have, using Lemma 2.1, p.279 of \cite{BC}, and \eqref{sigma-p},
\[
\begin{aligned}
\|\sqrt{P_{r-1}}A\|^2&=\tr(P_{r-1}A^2)\\
&=\sss_1\sss_r - (r+1)\sss_{r+1}\\
&=\binom{m}{1}\binom{m}{r}^{-\frac{1}{r+1}}\binom{m}{r}\binom{m}{r}^{-\frac{r}{r+1}}-(r+1)\binom{m}{r+1}\binom{m}{r}^{-1}\\
&=m-(m-r)=r,
\end{aligned}
\]
where we used that $(r+1)\binom{m}{r+1}=(m-r)\binom{m}{r}$ and the convention that $\binom{m}{r}=0$ if $r>m.$
The first result of this paper is the following gap theorem.

\begin{theorem}\label{theo-main-1}
Let $\Sigma^n$ be a complete, $n$-dimensional, properly immersed, self-shrinker of the $r$-mean curvature flow in $\R^{n+1},$ $1\leq r\leq n$. Suppose the $(r-1)$-th Newton transformation $P_{r-1}$ is positive semidefinite, bounded, and satisfies
\[
\|\sqrt{P_{r-1}}A\|^2\leq r,
\]
then $\Sigma^n$ is 
\begin{itemize}
\item[(i)] a hyperplane in $\R^{n+1}_{\raisepunct{,}}$ if $\|\sqrt{P_{r-1}}A\|^2<r;$
\item[(ii)] the round sphere $\s^n\left(\delta_{n}(r)\right)$ or the cylinder $\Sigma^n=\s^{m}\left(\delta_{m}(r)\right)\times\R^{n-m}$ in $\R^{n+1},$ $r\leq m\leq n-1,$ provided $P_{r-1}$ is positive definite. Here $\delta_{m}(r)=\binom{m}{r}^{\frac{1}{r+1}}$.
\end{itemize}
\end{theorem}
\begin{remark}Observe that for $r=1$, Theorem \ref{theo-main-1} is exactly  Cao and Li's result for hypersurfaces, see Theorem 1.1 of \cite{cao-li}, since $P_{r-1}=I$ is positive definite, bounded, $\|A\|^2\leq 1,$ and, for self-shrinkers of the mean curvature flow, properness is equivalent to have polynomial volume growth, see Theorem 1.3 of \cite{Cheng-Zhou}.
\end{remark}

\begin{corollary}[Cao-Li for hypersurfaces, \cite{cao-li}]\label{cao-li}
If $X\colon\Sigma^n\rightarrow\R^{n+1}$ is a complete  $n$-dimensional self-shrinker of the mean curvature flow, without boundary and with polynomial volume growth, and satisfies
\[
\|A\|^2\leq 1,
\]
then it is one of the following:
\begin{itemize}
\item[(i)] a round sphere $\s^n(\sqrt{n})$ in $\R^{n+1};$
\item[(ii)] a cylinder $\s^m(\sqrt{m})\times\R^{n-m}$ in $\R^{n+1},$ $1\leq m\leq n-1;$
\item[(iii)] a hyperplane in $\R^{n+1}.$
\end{itemize}
In particular, if $\|A\|^2<1,$ then $\Sigma^n$ is a hyperplane. Here, $\|A\|^2$ is the squared norm of the second fundamental form of $\Sigma^n.$
\end{corollary}

For $r=n,$ the hypersurface  $\Sigma^n$ is a self-shrinker of the Gaussian curvature flow. By \eqref{Pr-rec}, it holds $KI=P_{n-1}A,$ i.e., 
\begin{equation}\label{gauss-eigen}
K=\langle K e_i, e_i\rangle =\langle P_{n-1}A e_i, e_i\rangle =k_i\sss_{n-1}(A_i),
\end{equation}
for every $i=1,\ldots,n.$

We claim that $P_{n-1}$ be positive semidefinite is equivalent to a choice of orientation when $\Sigma^n$ is weakly convex, meaning, $A$ is positive semidefinite. Indeed, by \eqref{gauss-eigen}, if $P_{n-1}$ positive semidefinite then $\sss_{n-1}(A_i)\geq 0,$ for all $i=1,\ldots,n.$ This gives that each $k_i$ has the same sign of $K,$ in particular, they have the same sign. The converse is also true. 
On the other hand, since $\tr(P_{r-1}A^2)=\sss_1\sss_r-(r+1)\sss_{r+1}$ (see Lemma 2.1, p.279 of \cite{BC}), if $r=n,$ then $\sss_{n+1}=0$ and
\[
\|\sqrt{P_{n-1}}A\|^2=\tr(P_{n-1}A^2)=HK\geq0.
\]
Therefore, since $\tr(P_{n-1})=\sigma_{n-1},$ we have the following

\begin{corollary}\label{cor-Gauss}
Let $\Sigma^n$ be a complete, $n$-dimensional, properly immersed, weakly convex, self-shrinker of the Gaussian curvature flow in $\R^{n+1}$. If $\sigma_{n-1}$ is bounded and
\[
HK\leq n,
\]
then $\Sigma^n$ is one of the following:
\begin{itemize}
\item[(i)] the unitary round sphere $\s^n(1);$
\item[(ii)] a hyperplane in $\R^{n+1}.$
\end{itemize}
In particular, if $HK<n,$ then $\Sigma^n$ is a hyperplane in $\R^{n+1}$.
\end{corollary}


If we  remove the properness condition of the hypotheses of Theorem \ref{theo-main-1} we  obtain

\begin{theorem}\label{theo-main-2}
Let $\Sigma^n$ be a complete $n$-dimensional self-shrinker of the $r$-mean curvature flow in $\R^{n+1},$ for $1\leq r\leq n$. If the $(r-1)$-th Newton transformation $P_{r-1}$ is positive semidefinite,
\[
\sup \|A\|^2<\infty,\quad\mbox{and}\quad \sup\|\sqrt{P_{r-1}}A\|^2<r,
\]
then $\Sigma^n$ is a hyperplane in $\R^{n+1}$.
\end{theorem}
For $r=1$ we extend the result of Cheng and Peng for hypersurfaces, see Theorem 1.1 of \cite{Cheng-Peng}:
\begin{corollary}[Cheng-Peng for hypersurfaces, \cite{Cheng-Peng}]\label{CP1}
If $\Sigma^n$ is a complete $n$-dimensional self-shrinker of the mean curvature flow in $\R^{n+1},$ then one of the following holds:
\begin{itemize}
\item[(i)] $\sup \|A\|^2\geq 1;$
\item[(ii)] or $\|A\|=0$ and $\Sigma^n$ is a hyperplane in $\R^{n+1}$.
\end{itemize}
In particular, if $\sup \|A\|^2 <1,$ then $\Sigma^n$ is a hyperplane in $\R^{n+1}$.
\end{corollary}

\begin{remark}
Notice that the hypothesis $\|\sqrt{P_{r-1}}A\|^2\leq r$ in Theorem \ref{theo-main-1} and Theorem \ref{theo-main-2} does not give any natural bounds on the second fundamental form for $r>1$, unlike Cao-Li's and Cheng-Peng's results. This drives us to impose new barriers to control the geometry and obtain the classification. 
\end{remark}

\begin{remark}
Cheng and Zhou \cite{CZ2017}, see Corollary 4 proved that complete self-shrinkers (in arbitrary codimension) of the mean curvature flow whose principal curvatures satisfy $\sup_{1\leq i\leq n} k_i^2\leq \delta <1,$ for some constant $\delta\geq0,$ are properly immersed, have finite weighted volume, and have polynomial volume growth. Since $\sup_{1\leq i\leq n} k_i^2\leq\|A\|^2,$ if we assume that $\sup\|A\|^2<1,$ then,  taking $\delta=\sup\|A\|^2$ and using the result of Cheng and Zhou, we conclude that the self-shrinker in the hypothesis of the result of Cheng and Peng is indeed properly immersed. We also point out that the equivalence between properness and polynomial volume growth in \cite{CZ2017} holds in a more general context, see \cite{CZ2021}.
\end{remark}

Taking $r=n,$ then we obtain the following result for self-shrinkers of the Gaussian curvature flow: 
\begin{corollary}\label{cor-main-2}
Let $\Sigma^n$ be a $n$-dimensional, complete, weakly convex, self-shrinker of the Gaussian curvature flow. If 
\[
\sup \|A\|^2<\infty \quad \mbox{and}\quad \sup HK<n,
\]
then $\Sigma^n$ is a hyperplane in $\R^{n+1}$.
\end{corollary}

\begin{remark}
    Recently, Batista and Xavier proved in \cite{BX} results in the same direction of Theorems \ref{theo-main-1} and \ref{theo-main-2} assuming some  additional hypotheses, besides assuming weak convexity, i.e., the second fundamental form is positive semidefinite. They proved that, 
    \begin{itemize}
        \item[(i)] if $\Sigma^n$ is compact (without bondary), weakly convex and \[\tr(P_{r-1}A^2)\leq r,\quad 1\leq r\leq n,\] then $\Sigma^n$ is a sphere (Theorem A);
        \item[(ii)] if $\Sigma^n$ is complete, weakly convex, $\sss_1$ is bounded and \[\tr(P_{r-1}A^2)<r, \quad 1\leq r\leq n,\] then $\Sigma^n$ is a hyperplane in $\R^{n+1}$ (Theorem B).
    \end{itemize}
   Notice that Theorem A is an immediate corollary of  Theorem \ref{theo-main-1} item (i) and Theorem B is a corollary of Theorem \ref{theo-main-2}, since $A\geq0$ and $\sss_1$ bounded imply that all the principal curvatures are nonnegative and bounded, which gives that $P_{r-1}$ is positive semidefinite and bounded, but the converse is not necessarily true. 
\end{remark}

\begin{remark}\label{Pr-positive}
There are some conditions to deduce that $P_{r-1}$ is positive semidefinite on a connected hypersurface. In the following, we point out some of them:
\begin{itemize}
\item[(i)] if $\sss_{r}=0,$ then $P_{r-1}$ is semidefinite. If $r-1$ is odd, then we can choose an orientation such that $P_{r-1}$ is positive semidefinite and, if $r-1$ is even and $\sss_{r-1}\geq 0,$ then $P_{r-1}$ is positive semidefinite; 
\item[(ii)] if $\sss_{r}=0,$ and $\sss_{r+1}\neq 0,$ then $P_{r-1}$ is definite. If $r-1$ is odd, then we can choose an orientation such that $P_{r-1}$ is positive definite and, if $r-1$ is even and $\sss_{r-1}\geq 0,$ then $P_{r-1}$ is positive definite;
\item[(iii)] if $\sss_k>0$ for some $1\leq k\leq m-1$ and there exists a point where all the principal curvatures are nonnegative, then $P_r$ is positive definite for every $1\leq r\leq k-1.$
\end{itemize}
The proof of item (i) is a consequence of Lemma 1.1 and Equation (1.3) of \cite{HL1}, p.250-251, and a direct proof can be found in \cite{ASZ}, Proposition 2.4, p.188-189. In its turn, the proof of item (ii) can be found \cite{HL2}, Proposition 1.5, p.873, and the proof of item (iii) can be found in \cite{BC}, Proposition 3.2, p.280-281 (see also \cite{CR}, Proposition 3.2, p.188).
\end{remark}

This paper is organized as follows: in Section \ref{sec2} we prove Theorem \ref{theo-main-1} using techniques of parabolicity for a certain second-order differential operator which generalizes the drifted Laplacian, while Section \ref{sec3} is devoted to the proof of Theorem \ref{theo-main-2} by using an Omori-Yau type maximum principle. for the same differential operator.


\section{Proof of Theorem \ref{theo-main-1}}\label{sec2}

Let $X\colon \Sigma^n \rightarrow \mathbb{R}^{n+1}$ be a  hypersurface and $f\colon \Sigma^n\to\R$ be a smooth function. Define the second-order differential operator 
\begin{equation}\label{Lr}
L_{r} f = \tr(P_{r}\hess f), \quad 0\leq r\leq n-1,
\end{equation}
where $\hess f(v)=\nabla_{v}\nabla f$ is the hessian operator  and $\nabla f$ is the gradient of $f$ on $\Sigma^n.$ It can be proved that $L_{r} f = \di(P_{r}(\nabla f)),$  see Proposition B on page 470 of \cite{Reilly}. We also define drifted-$L_{r}$ operator by 
\begin{equation}\label{drifted-Lr}
\mathcal{L}_{r}f=L_{r} f-\langle X, \nabla f\rangle, \quad 0\leq r \leq n-1,
\end{equation}
where $X$ is the position vector field. 
\begin{definition}[Def. 4.2, \cite{AMR} p.243] The operator $\mathcal{L}_{r}$ is strongly parabolic on $\Sigma^n$ if for each nonconstant $u\in C^{2}(\Sigma^n)$ with $u^{\ast}=\sup_{\Sigma^n}u <+\infty$ and for each $\eta\in \mathbb{R}$ with $\eta < u^{\ast}$ we have \[\inf_{\Omega_{\eta}}\mathcal{L}_r(u)<0,\] where $\Omega_{\eta}=\{ x\in \Sigma^n\colon u(x)>\eta\}.$
\end{definition}

The Khasminskii Test (Theorem 4.12 of \cite{AMR}) gives sufficient conditions to guarantee strong parabolicity for the operator $\mathcal{L}_{r}$ on $\Sigma^n$ if $P_r$ is positive definite. However, in this article,  we mostly consider  positive semidefinite Newton transformations. In this case, following verbatim the  proof of the Kashminskii test in \cite{AMR} to $\mathcal{L}_{r}$ when $P_r$ is  positive semidefinite and we have the following statement.
\begin{proposition}\label{prop-3}
Assume the existence of a function $\gamma\in C^2(\Sigma^n)$ such that
\begin{equation}\label{khas}\left\{\begin{array}{llll}
\gamma(x)&\to& +\infty& \mbox{as}\,\, x\to\infty, \\
\mathcal{L}_r\gamma& < &0 &\mbox{off a compact set,}
\end{array}\right.
\end{equation}
where we are assuming that $P_r$ is positive semidefinite. If $u\in\CC^2(\Sigma^n)$ is not constant and satisfies $u^\ast=\sup_{\Sigma}u<\infty$, then $u$ achieves its maximum at a point $z_0\in\Sigma^n$ or 
\[
\inf_{B_\eta} \LL_r u <0,
\]
for every $0<\eta<u^\ast,$ where $B_\eta=\{x\in\Sigma^n; u(x)>u^\ast -\eta\}.$ In particular, if $\LL_r u\geq0$ and $u$ does not achieve its maximum, then $u$ is constant.  In addition, if $P_r$ is positive definite, then $\LL_r$ is strong parabolic $\Sigma^n.$
\end{proposition}
\begin{remark}In the proof of the Khasminskii test, the necessity to  $P_r$ to be positive definite is to show that  (see Theorem 3.10 of \cite{AMR}) that $u$ can not achieve its maximum at a finite point $z_0$.  
\end{remark}
 \begin{proof}
Assume that $u^{\ast}$ can not be achieved in any point $z_0 \in \Sigma^n.$ Let us prove that, given $u\in\CC^2(\Sigma^n)$ with $u^\ast>0$ and $0<\eta<u^\ast$ fixed, but arbitrary, it holds
\[
\inf_{B_\eta} \LL_r u <0,
\]
where $B_\eta=\{x\in\Sigma^n; u(x)>u^\ast -\eta\}.$ Suppose by contradiction that $\LL_{r}u \geq 0$ on $B_{\eta}$.
Let 
\begin{equation}\label{eq2.5b}
\Omega_t=\{x\in \Sigma\colon \gamma (x) >t\}
\end{equation}
and
\begin{equation}\label{eq2.5b-2}
\Omega_t^c=\{x\in \Sigma\colon \gamma (x) \leq t\}
\end{equation}
be its complement. Notice that, since $\gamma(x)\to\infty$ when $x\to\infty,$ then $\Omega_t^c$ is compact. In particular, there exists $u_t^\ast=\max_{\Omega_t^c}u(x).$ Notice that $\{\Omega_t^c\}_{t\in\R}$ is an exhaustion of $\Sigma^n,$ since
\[
\bigcup_{t\in\R} \Omega_t^c = \Sigma^n \quad \mbox{and}\quad \Omega_{t_1}^c \subset \Omega_{t_2}^c \quad \mbox{for} \quad t_1< t_2.
\]
Moreover, it holds $u_{t_1}^\ast\leq u_{t_2}^\ast$ if $t_1<t_2.$ Since $u^\ast$ is not achieved, there exists a divergent sequence $t_j\to\infty$ such that $u_{t_j}^\ast\to u^\ast.$ Thus, we can 
choose $T_1>0$ sufficiently large such that
\begin{equation}\label{est-1}
u_{T_1}^\ast>u^\ast-\frac{\eta}{2}.
\end{equation}
Now, let $\alpha\in\R$ such that
\begin{equation}\label{est-2}
u_{T_1}^\ast<\alpha< u^\ast.
\end{equation}
Since $u_{t_j}^\ast\to u^\ast,$ we can find $T_2>T_1$ such that 
\begin{equation}\label{est-3}
u_{T_2}^\ast>\alpha.
\end{equation}
Select $\bar{\eta}>0$ small enough in order to have
\begin{equation}\label{est-4}
\alpha+\bar{\eta}<u_{T_2}^\ast.
\end{equation}
For every $\delta>0$ small, define
\begin{equation}\label{gamma-delta}
\gamma_\delta(x)=\alpha+\delta(\gamma(x)-T_1).
\end{equation}
Since $\Omega_{t_1}\supset\Omega_{t_2}$ for $t_1<t_2,$ the function $\gamma_\delta$ satisfies the following properties:
\begin{itemize}
\item[(i)] $\gamma_\delta(x)=\alpha$ for every $x\in\partial\Omega_{T_1};$

\item[(ii)] $\LL_r\gamma_\delta = \delta\LL_r\gamma<0$ on $\Omega_{T_1}$ for $T_1$ large enough (by hypothesis);

\item[(iii)] $\alpha<\gamma_\delta(x)\leq \alpha+\delta(T_2-T_1)$ on $\Omega_{T_1}\backslash\Omega_{T_2},$ since $T_1<\gamma(x)\leq T_2$ on $\Omega_{T_1}\backslash\Omega_{T_2}.$
\end{itemize}
Choosing $\delta>0$ small enough such that $\delta(T_2-T_1)<\bar{\eta}$ and by using (iii), we have
\begin{equation}\label{est-5}
\alpha<\gamma_\sigma(x)<\alpha+\bar{\eta}\quad \mbox{on} \quad \Omega_{T_1}\backslash\Omega_{T_2}.
\end{equation}
Since 
\[
\gamma_\delta(x)=\alpha>u_{T_1}^\ast\geq u(x) \quad \mbox{on} \quad \partial\Omega_{T_1},
\]
we have
\begin{equation}\label{est-6}
(u-\gamma_\delta)(x)\leq 0 \quad \mbox{on} \quad \partial\Omega_{T_1}.
\end{equation}
On the other hand, since 
\[
\Omega_{T_1}\backslash\Omega_{T_2}=\{x\in\Sigma^n; T_1<\gamma(x)\leq T_2\}\subset \Omega_{T_2}^c
\]
and using the divergence of the sequence by taking $T_1$ large enough, there exists $\bar{x}\in\Omega_{T_1}\backslash\Omega_{T_2}$ such that $u(\bar{x})=u_{T_2}^\ast.$ This implies
\begin{equation}\label{x-bar}
\aligned
(u-\gamma_\delta)(\bar{x})&=u_{T_2}^\ast - \alpha -\delta(\gamma(x)-T_1)\\
&> u_{T_2}^\ast - \alpha -\delta(T_2-T_1)\\
&> u_{T_2}^\ast -\alpha -\bar{\eta}>0,
\endaligned
\end{equation}
where we used the definition of $\gamma_\delta,$ the fact that $\bar{x}\in\Omega_{T_1}\backslash\Omega_{T_2}$, \eqref{est-5}, and \eqref{est-4}. Notice that, since $u^\ast<\infty$ and $\gamma(x)\to\infty$ when $x\to\infty,$ it holds
\begin{equation}\label{min}
(u-\gamma_\delta)(x)<0 \quad \mbox{on} \quad \Omega_{T_3}
\end{equation}
for $T_3>T_2$ sufficiently large. Thus, by \eqref{x-bar} and \eqref{min} we conclude that there exists a positive maximum of $u-\gamma_\delta$ achieved at a point $z_0\in \overline\Omega_{T_1}\backslash\Omega_3.$ In particular, since $P_r$ is positive semidefinite, it holds
\[
\LL_r(u-\gamma_\delta)(z_0)\leq 0.
\]
But notice that $z_0\in B_\eta.$ Indeed, $z_0\in\Omega_{T_1}$ and
\[
\aligned
u(z_0)&>\gamma_\delta(z_0)=\alpha+\delta(\gamma(z_0)-T_1)\\
&>\alpha>u_{T_1}^\ast>u^\ast-\frac{\eta}{2}>u^\ast-\eta.
\endaligned
\]
Therefore, since $z_0\in B_\eta,$ it holds, at $z_0,$
\[
0\leq \LL_r u \leq \LL_r\gamma_\delta = \delta\LL_r\gamma<0.
\]
This contradiction concludes the proof. In particular, if  $u\in\CC^2(\Sigma^n)$ such that $\LL_r u\geq 0$ with  $u^\ast<\infty$ then  then either $u(z_0)=u^{\ast}$ for some $z_0\in \Sigma^{n}$ or $u$ must be a constant function. Moreover, if $P_r$ positive definite, then $\LL_r$ is an elliptic operator. Thus, by the generalized Hopf maximum principle Theorem 3.10 of \cite{AMR}, any $\LL_r$-subharmonic function $u\in\CC^2(\Sigma^n)$, bounded above, can not achieves its maximum unless it is constant. Therefore,  $u$ does not achieve its maximum and the rest of the proof implies that $\LL_r$ is strongly parabolic.
\end{proof}
Our next result shows that, under fairly mild geometric assumptions, $\Sigma^n$ satisfies Khasminskii's conditions \eqref{khas} for the operator  $\LL_{r-1}$.

\begin{proposition}\label{parab}
 Let $X\colon \Sigma^{n}\rightarrow \mathbb{R}^{n+1}$ be a complete properly immersed self-shrinker of the $r$-mean curvature flow. If there exists $0<c<1,$ such that   
 \begin{equation}\label{eq3.8}
 (n-r+1)\limsup_{x\to \infty}\frac{\sss_{r-1}(x)}{\Vert X(x)\Vert^2} \leq c,
 \end{equation} 
then the function $\gamma(x)=\|X(x)\|^2$ satisfies the Khasminskii's conditions \eqref{khas} of Proposition \ref{prop-3} for the operator $\LL_{r-1}$. In particular, if $u\in\CC^2(\Sigma^n)$ is bounded above and sastisfies $\LL_{r-1}u\geq 0,$ then $u$ achieves its maximum or $u$ is constant. Moreover, if $P_{r-1}$ is positive definite, then $\LL_{r-1}$ is strong parabolic $\Sigma^n.$

\end{proposition}

\begin{proof} 
Since the immersion is proper, the function $\gamma(x)=\|X(x)\|^2\to\infty$ when $x\to\infty.$ On the other hand,  using Lemma 1, p.208, of \cite{Alencar-Colares}, we have that
\[
\begin{aligned}
\frac{1}{2}L_{r-1}\|X\|^2 &= (n-r+1)\sss_{r-1}+r\sss_r\lan X,N\ran\\
&=(n-r+1)\sss_{r-1}-r\lan X,N\ran^2.\\
\end{aligned}
\]
This gives
\[
\begin{aligned}
\frac{1}{2}\LL_{r-1}\|X\|^2&=(n-r+1)\sss_{r-1}-r\lan X,N\ran^2 - \lan\n\|X\|^2,X\ran\\
&=(n-r+1)\sss_{r-1}-r\lan X,N\ran^2 - \|X^\top\|^2\\
&=(n-r+1)\sss_{r-1}-(r-1)\lan X,N\ran^2 - \|X\|^2\\
&\leq (n-r+1)\sss_{r-1} - \|X\|^2\\
&=\left[(n-r+1)\frac{\sss_{r-1}}{\|X\|^2}-1\right]\|X\|^2\\
&\leq (c-1)\Vert X\Vert^2 <0,
\end{aligned}
\]
outside a suitable compact set. 
\end{proof}

In the following lemma, we show that, for self-shrinkers of the $r$-mean curvature flow, $\sss_r$ satisfies a second-order partial differential equation, that is (semi-)elliptic if $P_{r-1}$ is positive (semi)definite:


\begin{proposition}\label{main-id}
Let $X\colon\Sigma^n\rightarrow\R^{n+1}$ be  a self-shrinker of the $r$-mean curvature flow, i.e.,  a hypersurface such that $\sss_r=-\lan X,N\ran.$ Then
\begin{equation}\label{eq-main-id}
\LL_{r-1}\sss_r+\left[\|\sqrt{P_{r-1}}A\|^2-r\right]\sss_r=0.
\end{equation}
Here, $N$ is the unit normal vector field of the immersion $X.$ Moreover, if $P_{r-1}$ is positive semidefinite and $\|\sqrt{P_{r-1}}A\|^2\leq r$, then 
\begin{equation}\label{Lr-squared-0}
\begin{aligned}
\frac{1}{2}\LL_{r-1}\sss_r^2&=\sss_r^2\left[r-\|\sqrt{P_{r-1}}A\|^2\right]+ \lan P_{r-1}(\n \sss_r),\n \sss_r\ran\geq 0.
\end{aligned}
\end{equation}
\end{proposition}
\begin{proof}
By Lemma 2, p. 209, of \cite{Alencar-Colares}, we have, for $1\leq r\leq n-1,$ 
\begin{equation}\label{main-id-1}
L_{r-1}\lan X,N\ran = -r\sss_r - (\sss_1\sss_r - (r+1)\sss_{r+1})\lan X,N\ran - \lan\n \sss_r,X\ran.
\end{equation}
Since $\Sigma^n$ satisfies $\sss_r=-\lan X,N\ran$ and by Lemma 2.1, p.279, of \cite{BC},
\[
\sss_1\sss_r - (r+1)\sss_{r+1}=\tr(P_{r-1}A^2)=\|\sqrt{P_{r-1}}A\|^2,
\] 
we obtain
\[
L_{r-1} \sss_r = r\sss_r - \|\sqrt{P_{r-1}}A\|^2\sss_r +\lan\n \sss_r,X\ran,
\]
i.e.,
\begin{equation}\label{main-id-2}
\LL_{r-1}\sss_r = -\left[\|\sqrt{P_{r-1}}A\|^2-r\right]\sss_r.
\end{equation}
On the other hand, since $L_{r-1}$ satisfies
\begin{equation}\label{Lr-fg}
L_{r-1}(fg)=fL_{r-1}g+gL_{r-1}f + 2\lan P_{r-1}(\n f),\n g\ran,
\end{equation}
it holds
\begin{equation}\label{LLr-fg}
\LL_{r-1}(fg)=f\LL_{r-1}g+g\LL_{r-1}f + 2\lan P_{r-1}(\n f),\n g\ran.
\end{equation}
Thus, by \eqref{eq-main-id} and \eqref{LLr-fg} we have 
\begin{equation}\label{Lr-squared}
\begin{aligned}
\frac{1}{2}\LL_{r-1}\sss_r^2&=\sss_r\LL_{r-1}\sss_r+\lan P_{r-1}(\n \sss_r),\n \sss_r\ran\\
&=\sss_r^2\left[r-\|\sqrt{P_{r-1}}A\|^2\right]+ \lan P_{r-1}(\n \sss_r),\n \sss_r\ran\geq 0.
\end{aligned}
\end{equation}
\end{proof}
\begin{lemma}
    Let $X\colon \Sigma^n \rightarrow \mathbb{R}^{n+1}$ be a  hypersurface and $f\colon \Sigma^n\to\R$ be a $C^{2}(\Sigma^n)$-function.  Suppose that $P_r$ is positive semidefinite and  $x_0$ is a point of maximum of $f$. Then  
\begin{equation}\label{Lr-b}
L_{r} f(x_0) = \tr(P_{r}\hess f(x_0)\leq 0.  
\end{equation}
\end{lemma}
\begin{proof}Let $\{e_1, \ldots, e_{n}\}$ be an orthonormal basis of $T_{x_0}\Sigma^n$ formed with eigenvalues of $P_r(x_0)$ with eigenvalues $0\leq \lambda_1\leq \lambda_2 \leq \cdots \leq \lambda_n$.
Then
\begin{eqnarray}
\tr(P_{r}\hess f)(x_0)&=&\sum_{i=1}^{n}\langle P_r \hess f(e_i), e_i\rangle \nonumber \\
&=&\sum_{i=1}^{n}\langle  \hess f(e_i), P_r e_i\rangle \nonumber \\
&=& \sum_{i=1}^{n} \lambda_i \hesss(e_i, e_i)(x_0)\nonumber\\
&\leq & 0\nonumber
\end{eqnarray} Since at a point of maximum $\hesss(e_i, e_i)(x_0)\leq 0$.
\end{proof}
\begin{proof}[Proof of Theorem \ref{theo-main-1}]
Using the Cauchy-Schwarz inequality for matrices, 
\[
[\tr(BC^t)]^2\leq \tr(BB^t)\tr(CC^t)
\]
for $B$ and $C$ matrices, where $(\ \ )^t$ denotes the transpose of a matrix, we have
\[
\aligned
\left[\tr(P_{r-1}A)\right]^2&=[\tr(\sqrt{P_{r-1}}(\sqrt{P_{r-1}}A))]^2\\
&\leq \tr(P_{r-1})\tr(P_{r-1}A^2),
\endaligned
\]
since $A,\ \sqrt{P_{r-1}},$ and $\sqrt{P_{r-1}}A$ are symmetric matrices that commute with each other. 
By hypothesis, $\|\sqrt{P_{r-1}}A\|^2=\tr(P_{r-1}A^2)\leq r$, $P_{r-1}$ is bounded and  by Lemma 2.1, p.279, of \cite{BC}, $\tr(P_{r-1}A)=r\sss_r,$ we have
\begin{equation}\label{bound-sr}
r^2\sss_r^2\leq \tr(P_{r-1})\tr(P_{r-1}A^2)\leq r\tr(P_{r-1})<\infty,
\end{equation}
i.e., $\sss_r^2$ is a bounded function.  Moreover, by Equation \eqref{Lr-squared-0}, $\LL_{r-1}\sigma^2_r\geq 0.$ 

Since $P_{r-1}$ is bounded and $\tr(P_{r-1})=(n-r+1)\sss_{r-1},$ it holds that $\sss_{r-1}$ is bounded. This gives that
\[
\limsup_{x\to\infty}\frac{\sss_{r-1}(x)}{\|X(x)\|^2}=0,
\]
provided $\Sigma^n$ is assumed to be properly immersed. Therefore, by Proposition \ref{parab}, p.\pageref{parab}, $\sss_r^2$ achieves a maximum point $x_0\in\Sigma^n$ or $\sss_r^2$ is constant. If $\Sigma^n$ is compact $\sss_r^2$ has a maximum  point.
If $\sss_r^2$ achieves its maximum at $x_0,$ then $\n\sss_r^2(x_0)=0$ and, by \eqref{Lr-squared-0} and \eqref{Lr-b},
\begin{equation}\label{max-point}
0\geq\frac{1}{2}\LL_{r-1}(\sss_r^2)(x_0)=\sss_r^2(x_0)\left[r-\|\sqrt{P_{r-1}}A\|(x_0)^2\right] \geq 0.
\end{equation}
Therefore, $\sss_r^2(x_0)=0$ or $\|\sqrt{P_{r-1}}A\|^2(x_0)=r.$ 

If $\|P_{r-1}A\|^2<r$ then $\sss_r^2\equiv 0$  since $\sss_r^2\geq 0$. Thus, $\lan X,N\ran=0$ and $\Sigma^n$ is a hyperplane.
On the other hand, if $P_{r-1}$ is positive definite, then, by Proposition \ref{parab}, $\LL_{r-1}$ is strongly parabolic. Therefore, since $\sigma_r^2$ is bounded and $\LL_{r-1}\sigma_r^2\geq 0,$ we can conclude that $\sigma_r^2$ is constant.
By Theorem 1 of \cite{DT}, the  hypersurfaces of $\R^{n+1}$ with constant support function $\lan X,N\ran$ are $\Sigma^n=\s^{m}(R)\times\R^{n-m},$ where $0\leq m\leq n,$ for an appropriate radius $R.$ Here we are considering that $\Sigma^n=\R^n$ is a hyperplane, for $m=0$, and $\Sigma^n=\s^n(R)$ is the round sphere, for $m=n$. Since the principal curvatures of $\s^{m}(R)\times\R^{n-m}$ are $k_1=1/R,$ with multiplicity $m,$ and $k_2=0,$ with multiplicity $n-m,$ we have that 
\begin{equation}\label{Sr-Raio}
\sigma_r=\binom{m}{r}\frac{1}{R^r}\raisepunct{,}
\end{equation}
where we are adopting the convention that $\binom{m}{r}=0$ if $r>m.$ Since, for $1\leq m\leq n,$ it holds $\lan X,N\ran=-R$ in these surfaces, from the self-shrinker equation $\sss_r=-\lan X,N\ran$ and using \eqref{Sr-Raio}, we obtain that
\begin{equation}\label{radius}
R=\binom{m}{r}^{\frac{1}{r+1}}.
\end{equation}
The Example \ref{example}, p.\pageref{example}, shows us the sphere (for $m=n$) and cylinders (for $1\leq m\leq n-1$) with radius given in \eqref{radius} satisfy $\|\sqrt{P_{r-1}}A\|^2=r.$


\end{proof}
\begin{remark}\label{remark-cao-li}
If $r=1,$ we have $P_{r-1}=I$ is naturally positive definite and $\LL_{r-1}=\LL:=\Delta - \lan X,\n\cdot\ran,$ the so called drifted Laplacian, is parabolic. Thus, under the hypothesis, we can conclude that $\sigma_1^2$ is constant.
This gives an alternative proof of Cao-Li result for hypersurfaces, see Corollary \ref{cao-li}, p.\pageref{cao-li}. 
\end{remark}

\begin{proof}[Proof of Corollary \ref{cor-Gauss}]
In the case $r=n,$ if $\sigma_n^2$ achieves a maximum at $x_0\in\Sigma^n,$ then, by \eqref{Lr-squared}, $\sss_n(x_0)^2=0$ or $\|P_{n-1}A\|^2(x_0)=H(x_0)K(x_0)=n.$ In the first case, we have that $\sss_n^2\equiv0,$ which gives that $\lan X,N\ran=0$ and $\Sigma^n$ is a hyperplane of $\R^{n+1}.$ In the second case, by \eqref{gauss-eigen}, p.\pageref{gauss-eigen}, $\sss_{n-1}(A_i)\neq 0$ at $x_0$ for every $i=1,\ldots,n.$ Thus, $P_{n-1}$ is positive definite in a neighborhood of $x_0$ and, by \ref{Lr-squared}, $\LL_{n-1}\sigma_n^2>0.$ Therefore, by the classical Hopf maximum principle, $\sss_n^2$ is constant. The results comes following the conclusion of the proof of Theorem \ref{theo-main-1}.
\end{proof}

\section{Proof of Theorem \ref{theo-main-2}}\label{sec3}
Let $\Sigma^n$ be a $n$-dimensional Riemannian manifold, $f:\Sigma^n\to\R$ be a class $\CC^2$ function, and $\phi:T\Sigma^n\to T\Sigma^n$ be a linear symmetric tensor. Define the second-order differential operator
\[
\mathcal{L}_{\phi} f := \tr(\phi\hess f) - \lan V,\n f\ran,
\] 
where $V$ is a vector field defined on $\Sigma^n$. 

The following maximum principle is a slight extension of Theorem 1, p.246, of \cite{BP} and we include a proof here for the sake of completeness.
\begin{lemma}\label{max-princ-intrinsic}
Let $\Sigma^n$ be an $n$-dimensional complete Riemannian manifold and $\phi:T\Sigma^n\to T\Sigma^n$ be a symmetric and positive semidefinite linear tensor. If there exists a positive function $\gamma\in \CC^2(\Sigma^n)$ such that
\begin{itemize}
\item[(i)] $\gamma(x)\to\infty$ when $x\to\infty;$
\item[]
\item[(ii)] $\displaystyle{\limsup_{x\to\infty}\mathcal{L}_\phi \gamma(x)<\infty};$
\item[]
\item[(iii)] $\displaystyle{\limsup_{x\to\infty}\|\n\gamma(x)\|<\infty},$
\end{itemize}
\vspace{2mm}
then, for every function $u\in\CC^2(\Sigma^n)$ satisfying
\begin{equation}\label{hyp-u}
\lim_{x\to\infty}\frac{u(x)}{\gamma(x)}=0,
\end{equation}
there exists a sequence of points $x_k\in\Sigma^n$ such that
\begin{equation}\label{eq-OY}
\|\n u(x_k)\|<\frac{1}{k} \quad \mbox{and} \quad  \mathcal{L}_{\phi} u(x_k)<\frac{1}{k}.
\end{equation}
Moreover, if instead of \eqref{hyp-u} we assume that $u^\ast=\sup_{\Sigma^n} u<\infty$, then
\[
\lim_{k\to\infty} u(x_k) = u^\ast.
\] 
\end{lemma}
\begin{proof}
Let
\[
f_k(x)=u(x)-\ve_k\gamma(x),
\]
for each positive integer $k,$ where $\ve_k>0$ is a sequence satisfying $\ve_k\to 0,$ when $k\to\infty.$ Since, for a fixed $x_0\in\Sigma^n,$ the sequence $\{f_k(x_0)\}_k$, defined by $f_k(x_0)=u(x_0)-\ve_k\gamma(x_0)$ is bounded, adding a positive constant to the function $u,$ if necessary, we may assume that $f_k(x_0)>0$ for every $k>0$. Notice that, by \eqref{hyp-u}, 
\[
\lim_{x\to\infty}\frac{f_k(x)}{\gamma(x)}=\lim_{x\to\infty}\frac{u(x)}{\gamma(x)}-\ve_k=-\ve_k< 0,
\] 
which implies that $f_k$ is non-positive out of a compact set $\Omega_{k}\subset\Sigma^n$ containing $x_0$. Thus, $f_k$ achieves its maximum at a point $x_k\in \Omega_{k}$ for each $k\geq 1$. Since
\[
\n f_k = \n u - \ve_k\n\gamma
\]
and
\[
\hesss f_k (v,v)= \hesss u (v,v) - \ve_k\hesss\gamma(v,v),
\]
we have, at $x_k,$ that
\begin{equation}\label{eqPeng1}
\n u(x_k)=\ve_k\n\gamma(x_k) \quad \mbox{and} \quad \hesss u(x_k)(v,v)\leq \ve_k\hesss\gamma(x_k)(v,v).
\end{equation}
First, observe that \eqref{eqPeng1} and hypothesis (iii) imply
\[
\|\n u(x_k)\|=\ve_k\|\n\gamma(x_k)\|\leq \ve_k C_0 <\frac{1}{k}
\] 
for some $C_0>0$ and $\ve_k<\frac{1}{kC_0}.$ On the other hand, letting   $\{e_1,\ldots,e_n\}$ be  an orthonormal frame formed with  eigenvectors of $\phi\colon T\Sigma^n\to T\Sigma^n,$ with  nonnegative eigenvalues $\lambda_1,\ldots,\lambda_n,$  we have
\[
\aligned
\mathcal{L}_{\phi}  u(x_k)&= \sum_{i=1}^n\lan\hess u(x_k)(e_i),\phi(e_i)\ran - \lan V(x_k),\n u(x_k)\ran\\
&= \sum_{i=1}^n\lambda_i\lan\hess u(x_k)(e_i),e_i\ran- \lan V(x_k),\n u(x_k)\ran\\
& = \sum_{i=1}^n\lambda_i\hesss u(x_k)(e_i,e_i)- \lan V(x_k),\n u(x_k)\ran\\
&\leq \ve_k\sum_{i=1}^n\lambda_i\hesss\gamma(x_k)(e_i,e_i) - \ve_k\lan V(x_k),\n \gamma(x_k)\ran\\
&=\ve_k\mathcal{L}_{\phi}\gamma(x_k)\\
&\leq \ve_k C_1<\frac{1}{k}\raisepunct{,}
\endaligned
\]
if we take $$\ve_k<\frac{1}{k\max\{C_0,C_1\}}\raisepunct{.}$$ 
\vspace{2mm}
If $=u^\ast=\sup_{\Sigma^n}u(x)<\infty,$ then, given an arbitrary integer $m>0,$ let $y_m\in\Sigma^n$ such that
\[
u(y_m)>u^\ast - \frac{1}{2m}.
\]
This gives
\[
\aligned
f_k(x_k)&=u(x_k)-\ve_k\gamma(x_k)\geq f_k(y_m)\\
&=u(y_m)-\ve_k\gamma(y_m)\\
&>u^\ast - \frac{1}{2m}-\ve_k\gamma(y_m),
\endaligned
\]
which implies
\[
\aligned
u(x_k)&>u^\ast - \frac{1}{2m}-\ve_k\gamma(y_m)+\ve_k\gamma(x_k)\\
&>u^\ast - \frac{1}{2m}-\ve_k\gamma(y_m).\\
\endaligned
\]
Now, choosing $k_m$ such that $\ve_{k_m}\gamma(y_m)<\frac{1}{2m},$ we obtain that
\[
u(x_{k_m})>u^\ast -\frac{1}{m}.
\]
Thus, by replacing $x_k$ by $x_{k_m}$ if necessary, we can conclude that 
\[
\lim_{k\to\infty}u(x_k)=u^\ast.
\]
\end{proof}

As an application of Lemma \ref{max-princ-intrinsic}, we have the

\begin{lemma}\label{COR-max-princ-2}
Let $\Sigma^n$ be an $n$-dimensional complete hypersurface of $\R^{n+1}$ such that $\sup_{\Sigma^n}\|A\|^2<\infty.$ If $P_{r-1}:T\Sigma^n\to T\Sigma^n$ is a positive semidefinite linear tensor, then, for every function $u\in\CC^2(\Sigma^n)$ bounded from above, there exists a sequence of points $x_k\in\Sigma^n$ such that
\begin{equation}\label{eq-OY}
\lim_{k\to\infty}u(x_k)=\sup_{\Sigma^n}u,\quad\|\n u(x_k)\|<\frac{1}{k} \quad \mbox{and} \quad \LL_{r-1} u(x_k)<\frac{1}{k}.
\end{equation}
\end{lemma}
\begin{proof}
Let us apply Lemma \ref{max-princ-intrinsic} to $\phi=P_{r-1},$ $V=X,$ the position vector of $\Sigma^n$ in $\R^{n+1},$ and $\gamma(x)=\log(\rho(x)+2),$ where $\rho(x)=\dist(x,x_0)$ is the geodesic distance of $\Sigma^n$ to a fixed point $x_0\in\Sigma^n.$ Let $\{e_1,\ldots,e_n\}$ be an orthonormal frame of principal directions of $\Sigma^n$ and denote by $\lambda_1,\ldots,\lambda_n,$ the eigenvalues of $P_{r-1}.$  Notice that, since the extrinsic distance is less than or equal to the intrinsic distance, we have $\|X(x)-X(x_0)\|\leq \rho(x)<\rho(x)+2.$ This gives
\[
\frac{\|X(x)\|}{\rho(x)+2} \leq \frac{\|X(x)-X(x_0)\|}{\rho(x)+2}+\frac{\|X(x_0)\|}{\rho(x)+2}\leq 1 + c_0,
\] 
where $c_0=\sup_{\Sigma}\frac{\|X(x_0)\|}{\rho(x)+2}.$ Moreover,  by the Gauss equation,
\[
K(e_i\wedge e_j)=\lan A(e_i),e_i\ran \lan A(e_j),e_j\ran - \lan A(e_i),e_j\ran^2\geq -2\|A\|^2 \geq -C,
\]
where $C:=2\sup_{\Sigma^n}\|A\|^2.$ Since, for points outside of the cut locus of $x_0,$
\begin{equation}\label{grad-hess-rho}
\n\gamma(x)=\frac{\n\rho(x)}{\rho(x)+2} \quad \mbox{and} \quad \hesss\gamma(x)(v,v)=\frac{\hesss\rho(x)}{\rho(x)+2}-\frac{\lan\n\rho(x),v\ran^2}{(\rho(x)+2)^2}\raisepunct{,}
\end{equation}
we have that 
\[
\limsup_{x\to\infty}\|\n\gamma(x)\|=\limsup_{x\to\infty}\frac{1}{\rho(x)+2}=0
\] 
and, by the hessian comparison theorem, 
\[
\aligned
\mathcal{L}_{r-1}\gamma(x)&=\sum_{i=1}^n \lambda_i\hesss \gamma(x)(e_i,e_i) - \lan X(x),\n\gamma(x)\ran\\
&=\sum_{i=1}^n \lambda_i\left[\frac{\hesss\rho(x)(e_i,e_i)}{\rho(x)+2} -  \frac{\lan\n\rho(x),e_i\ran^2}{(\rho(x)+2)^2}\right] - \frac{\lan X(x),\n\rho(x)\ran}{\rho(x)+2}\\
&\leq\sum_{i=1}^n \lambda_i\frac{\hesss\rho(x)(e_i,e_i)}{\rho(x)+2} - \frac{\lan X(x),\n\rho(x)\ran}{\rho(x)+2}\\
&\leq \frac{\sqrt{C}\coth(\sqrt{C}\rho(x))}{\rho(x)+2}\sum_{i=1}^n \lambda_i[\lan e_i,e_i\ran - \lan\n \rho(x),e_i\ran^2] + \frac{\|X(x)\|}{\rho(x)+2}\\
&\leq\frac{2\sqrt{C}\tr(P_{r-1}(x))}{\rho(x)+2} +1+c_0<\infty,\\
\endaligned
\]
where we used that $P_{r-1}$ is positive semidefinite and bounded and that $\coth(\sqrt{C}\rho)<2$ for $\rho\gg 1$. For points in the cut locus of $x_0$ we use the Calabi trick as it was done by Cheng and Yau in \cite{CY}, p.341-342. The result then follows from Lemma \ref{max-princ-intrinsic}.
\end{proof}


We conclude the paper with the proof of Theorem \ref{theo-main-2}:

\begin{proof}[Proof of Theorem \ref{theo-main-2}]
If $\sup(\|\sqrt{P_{r-1}}A\|^2)\geq r$ there is nothing to prove. If $\sup\|\sqrt{P_{r-1}}A\|^2<r,$ then, by \eqref{bound-sr}, p.\pageref{bound-sr},
\[
\aligned
r^2\sss_r^2&\leq \left[\tr(P_{r-1}A)\right]^2\\
&\leq \tr(P_{r-1}A^2)\tr(P_{r-1})\\
&<r\tr(P_{r-1})<\infty,
\endaligned
\]
since  $\|A\|$ is bounded $P_{r-1}$ is bounded. Using Lemma \ref{COR-max-princ-2} in \eqref{Lr-squared-0}, p.\pageref{Lr-squared-0}, we have
\[
\aligned
0&\geq \limsup \LL_{r-1}\sss_r^2\\
&=\sup \sss_r^2\sup[r-\|\sqrt{P_{r-1}}A\|^2]\\
&\geq \sup \sss_r^2 [r-\sup\|\sqrt{P_{r-1}}A\|^2]\\
&\geq 0.
\endaligned
\]
This gives
\[
\sup \sss_r^2=0,
\]
i.e., $-\lan X,N\ran=\sss_r=0$ and, thus $\Sigma^n$ is a hyperplane.
\end{proof}


\end{document}